\documentclass[12pt, reqno]{amsart}
\usepackage{amsmath, amsthm, amscd, amsfonts, amssymb, graphicx, color}
\usepackage[bookmarksnumbered, colorlinks, plainpages]{hyperref}
\hypersetup{colorlinks=true,linkcolor=red, anchorcolor=green, citecolor=cyan, urlcolor=red, filecolor=magenta, pdftoolbar=true}
\newtheorem{theorem}{Theorem}[section]
\newtheorem{lemma}[theorem]{Lemma}
\newtheorem{proposition}[theorem]{Proposition}
\newtheorem{corollary}[theorem]{Corollary}
\theoremstyle{definition}
\newtheorem{definition}[theorem]{Definition}

\theoremstyle{remark}
\newtheorem{remark}[theorem]{Remark}

\usepackage{mathtools}

\begin{document}

\setcounter{page}{1}

\title[semi-uo-convergent in ordered vector spaces]{semi unbounded order convergent in ordered vector spaces}
\author[M. Ebrahimzadeh]{Masoumeh Ebrahimzadeh}
\address[M. Ebrahimzadeh]{Islamic Azad University, Department of  Mathematics, Sarab Branch, Sarab, Iran}
\email{ebrahimzadeh.math@yahoo.com}

\author[K. Haghnejad Azar]{Kazem Haghnejad Azar$^{*}$}
\address[K. Haghnejad Azar]{Department  of  Mathematics  and  Applications, Faculty of  Sciences, University of Mohaghegh Ardabili, Ardabil, Iran.}
\email{haghnejad@uma.ac.ir}
\keywords{ ordered vector space, pre-Riesz space, unbounded order convergent, semi unbounded order convergent.
\newline \indent $^{*}$Corresponding author}
\begin{abstract}
	Let $X$ be an ordered vector space. The net $\{x_\alpha\}\subseteq X$ is semi unbounded order convergent to $x$ (in symbol $x_\alpha\xrightarrow{suo}x$), if  there is a net $\{y_\beta\}$, possibly over a different index set, such that $y_\beta \downarrow 0$ and for every $\beta$ there exists $\alpha_0$ such that
	$\{\{\pm(x_\alpha - x)\}^u,y\}^l\subseteq \{y_\beta\}^l$, whenever $\alpha \geq \alpha_0$ and for all $0\leq y \in X$.
In vector lattice $E$,  semi unbounded order convergence is equivalent with unbounded order convergence.	
	We study some properties of this convergence and some of its relationships with others known order convergence.  
\end{abstract}
\maketitle

\section{Introduction}
First  we will ours motivate  to write this article. By studying of \cite{m}, we became interested in working on ordered vector spaces. In this article, we have discussed convergence in ordered vector spaces that similar to the unbounded order convergent that has been done in many articles, including \cite{4} and \cite{4b}. 
 
Here we bring some definitions of need.
Let $X$ be a real vector space and $K$ be a cone in $X$, that is, $K$ is a wedge ($x,y \in K$ and $\lambda,\mu \geq 0$ imply $\lambda x + \mu y \in K$) and $K \cap (-K) = \{0\}$. In $X$ a partial order is defined by $x\leq y $ whenever $y-x \in K$. The space $(X, K)$ (or, loosely $X$) is then called a (partially) ordered vector space.
 A subspace $M \subseteq X$ is majorizing in $X$ if for every $x\in X$ there exists some $m\in M$ with $x\leq m$ (or, equivalently, if for each $x\in X$ there exists some $m\in M$ with $m\leq x$).
 A subspace $M\subseteq X$ is called directed if for every $x,y \in M$ there is an element $z\in M$ such that $x\leq z$ and $y\leq z$. An ordered vector space $X$ is directed if and only if $X_+ $ is generating in $X$, that is, $X = X_+ - X_+$.
An ordered vector space $X$ is called Archimedean if for every $x,y \in X$ with $nx\leq y$ for every $n\in \mathbb{N}$ one has $x\leq 0$. 
The ordered vector space $X$ has the Riesz decomposition property (RDP) if for every $x_1,  x_2,  z \in K$ with $z\leq x_1 + x_2$ there exist $z_1,  z_2 \in K$ such that $ z = z_1 +z_2$ with $z_1 \leq x_1$ and $z_2 \leq x_2$. We call a linear subspace $M$ of an ordered vector space $X$ order dense in $X$ if for every $x\in X$ we have 
\begin{equation*}
x = \inf \{z\in M : x \leq z\},
\end{equation*}
that is, the greatest lower bound of set $\{z\in M : x \leq z\}$ exists in $X$ and equals to $x$, see page $360$ of \cite{1a}. Clearly, if $M$ is order dense in $X$, then $M$ is majorizing in $X$.
 Denote for a subset $M$ of $X$, the set of all upper bounds (resp, down bounds) by $M^u = \{x\in X : x\geq m \ for \ all\ m \in M\}$ (resp, $M^l = \{x\in X : x\leq m \ for \ all\ m \in M\}$).
  It is clear that for every subset $A,B$ of $X$, $(A^l + B^l) \subseteq (A+B)^l$. Moreover, if $X$ has Riesz decomposition property (in short RNP)  and $A,B \subseteq K$, then $(A+B)^l \subseteq A^l + B^l$.\\
   The elements $x,y \in X$ are called disjoint, in symbols $x \perp y$, if $\{\pm (x+y)\}^u = \{\pm (x-y)\}^u$. The disjoint complement of a subset $M\subseteq X$ is $M^d = \{x\in X \mid \forall y\in M: x \perp y\}$. A sequence $\{x_n\}\subseteq X$ is said to be disjoint, if for every $n \neq m$, $x_n \perp x_m$. 
    A linear subspace $B$ of an ordered vector space $X$ is called a band in $X$ if $B = B^{dd} $. A subset $M$ of an ordered vector space $X$ is called solid if for every $x\in X$ and $y\in M$ the relation $\{\pm y\}^u \subseteq \{\pm x\}^u$ implies $x\in M$. A solid subspace $M$ of $X$ is called an ideal.\\
 Recall that a linear map $i : X \rightarrow Y$ is said to be bipositive if for every $x\in X$ one has $i(x)\geq 0$ if and only if $x\geq 0$. A partially ordered vector space $(X,K)$ is called pre-Riesz space if for every $x,y,z \in X$ the inclusion $\{x+y,x+z\}^u \subseteq \{y,z\}^u$ implies $x\in K$. Clearly, each vector lattice is a pre-Riesz space, since the inclusion in definition of pre-Riesz space reduces to inequality $(x+y)\vee (x+z)\geq y \vee z$, so $ x+(y \vee z) \geq y \vee z$, which implies $x\geq 0$. By Theorem 4.3 of \cite{3},  partially ordered vector space $X$ is a pre-Riesz space if and only if there exist a vector lattice $Y$ and a bipositive linear map $i: X \rightarrow Y$ such that $i(X)$ is order dense in $Y$. The pair $(Y,i)$ (or, loosely $Y$) is then called a vector lattice cover of $X$. The theory of pre-Riesz spaces and their vector lattice covers is due to van Haandel, see \cite{5}.
A net $\{x_\alpha\} \subseteq X$ is said to be decreasing (in symbols, $x_\alpha \downarrow$), whenever $\alpha \geq \beta $ implies $x_\alpha \leq x_\beta$. For $x\in X$ the notation $x_\alpha \downarrow x $ means that $x_\alpha \downarrow $ and $\inf_\alpha \{x_\alpha\} = x$ both hold. The meanings of $x_\alpha \uparrow x$ are analogous.  We say that a net $\{x_\alpha\}\subseteq X$ $(o)$-converges (resp, $\tilde{o}$-converges) to $x\in X$ (in symbols, $x_\alpha \xrightarrow{o}x$, resp, $x_\alpha\xrightarrow{\tilde{o}}x$), if there is a net $\{y_\alpha\}\subseteq X$ (resp, $\{y_\beta\}$ possibly over a different index set) such that $y_\alpha \downarrow 0$ (resp, $y_\beta\downarrow 0$) and for all $\alpha$ (resp, for every $\beta$ there exists $\alpha_0$ that for all $\alpha \geq \alpha_0$ ) one has $ \pm(x_\alpha - x) \leq y_\alpha$, (resp, $ \pm(x_\alpha - x) \leq y_\beta$). 
 For two elements $y,z \in K$ with $ y \leq z$ denote the according order interval by $ [y,z] = \{x\in X: y\leq x \leq z\}$. A set $M\subset X$ is called order bounded if there are $y,z \in X$ such that $M\subseteq [y,z]$.  Let $X$ and $Z$ be ordered vector spaces. 
 
Recall that the net $\{x_\alpha\}$ in vector lattice $E$ is said to be unbounded order convergent to $x$ (or, $uo$-convergent for short) to $x$ if  there is a net $\{y_\beta\}$, possibly over a different index set, such that $y_\beta \downarrow 0$ and for every $\beta$ there exists $\alpha_0$ such that
$\mid x_\alpha - x\mid \wedge y \leq y_\beta$, whenever $\alpha \geq \alpha_0$ and for all $0\leq y \in E$. In this case, we write $x_\alpha\xrightarrow{uo}x$. 
\section{semi unbounded order convergence on ordered vector spaces}
\begin{definition}\label{ds}	Let $X$ be an ordered vector space. The net $\{x_\alpha\}\subseteq X$ is semi unbounded order convergent (or, $semi$-$uo$-convergent for short) to $x$ if  there is a net $\{y_\beta\}$, possibly over a different index set, such that $y_\beta \downarrow 0$ and for every $\beta$ there exists $\alpha_0$ such that
	 $\{\{\pm(x_\alpha - x)\}^u,y\}^l\subseteq \{y_\beta\}^l$, whenever $\alpha \geq \alpha_0$ and for all $0\leq y \in X$. In this case, we write $x_\alpha\xrightarrow{suo}x$. 
\end{definition}
\begin{remark}
${\tilde{o}}$-convergence implies ${suo}$-convergence, but the convergence, in general, not holds.
	Let $\{x_\alpha\}\subseteq X$ and $x_\alpha\xrightarrow{\tilde{o}}x$. Therefore, there exists a net $\{y_\beta\}\subseteq X$ and $y_\beta \downarrow 0$ and for all $\beta$, there is an $\alpha_0$ that for each $\alpha$, $\pm (x_\alpha -x) \leq y_\beta$. Let $z \in  \{\{\pm(x_\alpha - x)\}^u,y\}^l$ for all $0\leq y \in X$. Since for $\alpha \geq \alpha_0$, $\pm(x_\alpha - x)\leq y_\beta$, hence $z\leq y_\beta$. So  $\{\{\pm(x_\alpha - x)\}^u,y\}^l\subseteq \{y_\beta\}^l$. Therefore, $x_\alpha\xrightarrow{suo}x$.
	
On the other hand, the standard basis of $c_{0}$, $\{ e_n\}_{n=1}^{\infty}$ is ${suo}$-convergence, but not ${\tilde{o}}$-convergence.
	 	 Note that for order bounded nets, $semi$-$uo$-convergence is equivalent to $\tilde{o}$-convergence. 
\end{remark} 
\begin{proposition}\label{dd}
If $E$ is a vector lattice, then the net $\{x_\alpha\}\subseteq E$ is $uo$-convergent to $x$ iff is $semi$-$uo$-convergent to $x$.
\end{proposition}
\begin{proof}
 Let $\{x_\alpha\}\subseteq E$ and $x_\alpha\xrightarrow{uo}x$. There exists a net $\{y_\beta\}\subseteq E$ that $y_\beta\downarrow 0$ and for each $\beta$ there exists an $\alpha_0$ that for each $\alpha\geq \alpha_0$ we have $|x_\alpha - x| \wedge y \leq y_\beta$ for all $y\in E^+$. Let $z\in \{\{\pm(x_\alpha - x)\}^u,y\}^l$. It is clear that $z\leq |x_\alpha - x|\wedge y$ and therefore for each $\alpha\geq \alpha_0$, $z\leq y_\beta$. Therefore, $\{\{\pm(x_\alpha - x)\}^u,y\}^l\subseteq \{y_\beta\}^l$. It means that $x_\alpha\xrightarrow{suo}x$. 

For conversely, let  $x_\alpha\xrightarrow{suo}x$. It is clear that $(|x_\alpha - x| \wedge y) \in \{\{\pm(x_\alpha - x)\}^u,y\}^l$ for all $y\in E^+$. 
  Therefore, for each $\beta$ there exists $\alpha_0$ such that for each $\alpha \geq \alpha_0$, $(|x_\alpha - x| \wedge y )\in \{y_\beta\}^l$. Hence
   $|x_\alpha - x| \wedge y \leq y_\beta$ for all $y\in E^+$ whenever $\alpha \geq \alpha_0$. It means that $x_\alpha\xrightarrow{uo}x$.
\end{proof}
\begin{lemma}\label{eli}
	Let $X$ be an ordered vector space and $\{x_\alpha\}\subseteq X$, then
\end{lemma}
\begin{enumerate}
	\item $x_\alpha \xrightarrow{suo}x$ iff $(x_\alpha - x) \xrightarrow{suo}0$.
	\item if for each $\alpha$, $ x_\alpha \leq y$ and $x_\alpha \xrightarrow{suo}x$, then $ x \leq y$.
		\item if $0\leq x_\alpha \xrightarrow{suo}x$, then $0 \leq x$. 

	Moreover, if $X$ has the $RDP$ property, then 
		\item if $x_\alpha \xrightarrow{suo}x$ and $y_\alpha \xrightarrow{suo}y$, then $ \lambda x_\alpha + \mu y_\alpha \xrightarrow{suo} \lambda x + \mu y$ for each scalar $ \lambda,  \mu  \in\mathbb{R}$.
	\item if $x_\alpha \xrightarrow{suo}x$, $z_\alpha \xrightarrow{suo}z$ and $ x_\alpha \leq z_\alpha$ for all $\alpha$, then $ x \leq z$.\label{pp}
	\item if $x_\alpha\xrightarrow{suo}x$ and $x_\alpha \xrightarrow{suo}y$, then $x = y$.
	\item if  $x_\alpha \xrightarrow{suo}0$, then for each subnet $\{z_\gamma\}$ of $\{x_\alpha\}$, $z_\gamma\xrightarrow{suo}0$.
\end{enumerate}
\begin{proof}
	\begin{enumerate}
	\item By Definition \ref{ds}, it is established.
		\item By assumption, it is clear that $\{x_\alpha\}$ is order bounded. Therefore, $x_\alpha\xrightarrow{\tilde{o}}x$. So by Lemma 2.1 of \cite{m}, the proof is clear.
	\item	Let $0\leq x_\alpha \xrightarrow{suo}x$, therefore there exists net $\{y_\beta\}\subseteq X$ such that $y_\beta \downarrow 0$ and for every $\beta $ there exists $\alpha_0$ such that $\{  \{\pm (x_\alpha - x)\}^u, y\}^l \subseteq \{y_\beta\}^l$ whenever $ \alpha \geq \alpha_0$ and for all $0\leq y\in X$. We know that $\{x_\alpha - x\}\subseteq \{\pm(x_\alpha - x)\}$ for all $\alpha$. Therefore, it is clear that for all $\alpha$, $\{\pm(x_\alpha - x)\}^u \subseteq \{x_\alpha - x\}^u$. So we have
	 $\{\{x_\alpha - x\}^u,  y\}^l \subseteq \{\{\pm(x_\alpha -x)\}^u,  y\}^l$ for all $\alpha$. Therefore, for $\alpha \geq \alpha_0$,  $\{\{x_\alpha - x\}^u,  y\}^l\subseteq \{y_\beta\}^l$ for all $0\leq y\in X$. It is clear that $-x\in  \{\{x_\alpha - x\}^u,  y\}^l$. Hence $-x\leq y_\beta$. So $x\geq 0$.
		\item Let $\lambda > 0$ and $x_\alpha\xrightarrow{suo}x$. We have $\{ \{\pm(\lambda x_\alpha -\lambda x)\}^u, y\}^l = \lambda(\{  \{\pm(x_\alpha - x)\}^u, y^\prime\}^l)\subseteq \{\lambda y_\beta\}^l$. Note that if $\lambda < 0$, then $\lambda (\pm x_\alpha) = -\lambda (\pm x_\alpha)$. 
		
		We have $\{\{\pm (x_\alpha - x)\}^u + \{\pm(y_\alpha-y)\}^u\}\subseteq \{\pm(x_\alpha -x +y_\alpha -y)\}^u$. Therefore, $\{\{\pm(x_\alpha -x +y_\alpha -y)\}^u, y \}^l \subseteq \{ \{\{\pm (x_\alpha - x)\}^u + \{\pm(y_\alpha-y)\}^u\},  y\}^l$ for all $0\leq y\in X$.
		 Since $X$ has the $RDP$ property, $\{\{\pm(x_\alpha -x +y_\alpha -y)\}^u, y \}^l\subseteq  \{ \{\{\pm (x_\alpha - x)\}^u,  y\}^l + \{\{\pm(y_\alpha-y)\}^u\},  y\}^l$. Hence $x_\alpha + y_\alpha \xrightarrow{suo}x+y$.
\item For all $y\in X^+$ we have $\{\{x-z\}^u,  y\} =\{ \{x-x_\alpha +x_\alpha - z\}^u,  y\}$. It is clear that $\{ \{x-x_\alpha +z_\alpha - z\}^u,  y\}\subseteq \{ \{x-x_\alpha +x_\alpha - z\}^u,  y\}$ and therefore $\{ \{x-x_\alpha +x_\alpha - z\}^u,  y\}^l \subseteq \{ \{x-x_\alpha +z_\alpha - z\}^u,  y\}^l$. Since $X$  has the $RDP$ property, the rest of the proof is clear.
\item Similar to the proof of \ref{pp}, since for each $\alpha$,  we have $x_\alpha \leq x_\alpha$, therefore $x\leq y$ and $y\leq x$. Therefore, $x=y$.
\item We have $\{z_\gamma\}\subseteq \{x_\alpha\}$. Therefore, $\{x_\alpha\}^u \subseteq \{z_\gamma\}^u$ and so $\{\{z_\gamma\}^u,  y\}^l \subseteq \{\{x_\alpha\}^u,  y\}^l$ for each $0 \leq y \in X$. Hence the proof is complete.
	\end{enumerate}
\end{proof}
\begin{proposition}
		Let $B$ be a projection band of ordered vector space $X$ and $ P_{B}$
	the corresponding band projection. If $\{x_\alpha\}\subseteq X$ is $semi$-$uo$-null in $X$, then $\{P_B(x_\alpha)\}$ is $semi$-$uo$-null in $B$.
\end{proposition}
\begin{proof}
	Let $x_\alpha\xrightarrow{suo}0$ in $X$, therefore there is a net $\{y_\beta\}$, possibly over a different index set, such that $y_\beta \downarrow 0$ and for every $\beta$ there exists $\alpha_0$ such that
	$\{\{\pm(x_\alpha )\}^u,y\}^l\subseteq \{y_\beta\}^l$, whenever $\alpha \geq \alpha_0$ and for all $0\leq y \in X$. We know that $\pm P_B(x_\alpha)\leq \pm (x_\alpha)$ for all $\alpha$. Therefore, $\{\pm(x_\alpha)\}^u \subseteq \{\pm(P_B(x_\alpha))\}^u$.
	  So we have $\{\{\pm(P_B(x_\alpha))\}^u,  P_B(y)\}^l \subseteq \{\{\pm(x_\alpha)\}^u,  y\}^l\subseteq \{y_\beta\}^l$. Note that $P_B(y_\beta)\downarrow 0$ in $B$ and $P_B(y_\beta)\leq y_\beta$ for all $\beta$. Hence $P_B(y_\beta)\in \{y_\beta\}^l$. It means that for each $z\in X$ that $z\in\{\{\pm(P_B(x_\alpha))\}^u,  P_B(y)\}^l$, we have $z\leq (y_\beta)$. It is obvious that $P_B(z)\in \{\{\pm(P_B(x_\alpha))\}^u,  P_B(y)\}^l$, $P_B(z)\leq z$ and therefore $P_B(z)\in \{P_B(y_\beta)\}^l$. 
	  Therefore, 
	  $\{\pm(P_B(x_\alpha))\}^u, P_B(y)\}^l\subseteq\{P_B(y_\beta)\}^l$ and so $P_B(x_\alpha)\xrightarrow{suo}0$.
\end{proof}
\begin{definition}
 A net $\{x_\alpha\}$ in  ordered vector space $X$ is said to be $ suo$-Cauchy, if $\{x_\alpha - x_\beta\}_{(\alpha,  \beta)}$ $semi$-$uo$- converges to $0$ in $X$.
 \end{definition}
\begin{proposition}\label{abc}
		  Let $X$ be an ordered vector space with the property $RDP$.  Then
\begin{enumerate}
\item  if an $semi$-$uo$-Cauchy net $\{x_\alpha\}$ has an $semi$-$uo$-convergent subnet whose $semi$-$uo$-limit is $x$, then $x_\alpha \xrightarrow{suo}x$. 
\item each $semi$-$uo$-convergent net is a $semi$-$uo$-Cauchy net.
\end{enumerate}
\end{proposition}
\begin{proof}
	\begin{enumerate}
\item		 Let $\{x_\alpha\}$ be a $semi$-$uo$-Cauchy net in $X$ and $\{z_\gamma\}$ be a subnet of $\{x_\alpha\}$ such that $z_\gamma \xrightarrow{suo}x$. We have $x_\alpha - x = x_\alpha - z_\gamma + z_\gamma -x $ and $-(x_\alpha -x ) = -(x_\alpha - z_\gamma + z_\gamma -x)$.  
		 	 We have $\{\{\pm (x_\alpha - z_\gamma)\}^u + \{\pm(z_\gamma-x)\}^u\}\subseteq \{\pm(x_\alpha -z_\gamma + z_\gamma - x)\}^u$. Therefore, $\{\{\pm(x_\alpha - x)\}^u,y\}^l = \{\{\pm(x_\alpha -z_\gamma + z_\gamma -x)\}^u, y \}^l \subseteq \{ \{\{\pm (x_\alpha - z_\gamma)\}^u + \{\pm(z_\gamma - x)\}^u\},  y\}^l$ for all $0\leq y\in X$.
		 Since $X$ has the $RDP$ property, $\{\{\pm(x_\alpha -x)\}^u, y \}^l\subseteq  \{ \{\{\pm (x_\alpha - z_\gamma)\}^u,  y\}^l + \{\{\pm(z_\gamma - x)\}^u\},  y\}^l$. Because by assumption, $x_\alpha - z_\gamma\xrightarrow{suo}0$ and $z_\gamma\xrightarrow{suo}x$, the proof is complete.
\item Without loss generality, by Lemma \ref{eli} we assume that  $\{x_\alpha\}$ be a $semi$-$uo$-convergent to $0$  in $X$.  Then there is a net $\{y_\sigma\}$ such that $y_\sigma \downarrow 0$ and for every $\sigma$ there exists $\alpha_0$ such that
	 $\{\{\pm(x_\alpha )\}^u,y\}^l\subseteq \{y_\sigma\}^l$, whenever $\alpha \geq \alpha_0$ and for all $0\leq y \in X$. Obviously that 
	 $$  \{\pm(x_\alpha )\}^u  +   \{\pm(x_\beta )\}^u   \subseteq       \frac{1}{2} \{\pm(x_\alpha-x_\beta )\}^u.$$
Let $y>0$. By property $RDP$, we have
$$  \{ \frac{1}{2} \{\pm(x_\alpha-x_\beta )\}^u ,y\}^l  \subseteq  \{ \{\pm(x_\alpha )\}^u ,y\}^l  +   \{\{\pm(x_\beta )\}^u  ,y\}^l    .$$	 
It follows that  $\{ \frac{1}{2} \{\pm(x_\alpha-x_\beta )\}^u ,y\}^l \subseteq   2\{y_\sigma\}^l$, and so proof follows.

\end{enumerate}		 
\end{proof}
\begin{theorem}\label{55}
	Let $X$ be an order dense subspace of an ordered vector space $Y$ and $\{x_\alpha\}\subseteq X$. $x_\alpha\xrightarrow{suo}x$ in $X$ iff $x_\alpha\xrightarrow{suo}x$ in $Y$.
\end{theorem}
\begin{proof}
Let $x_\alpha\xrightarrow{suo}x$ in $X$ and $z\in \{\{\pm(x_\alpha - x)\}^u,y\}^l$ for all $0\leq y\in Y$. Therefore, $z\in \{\{\pm(x_\alpha - x)\}^u,y\}^l$ for all $0\leq y\in X$. By assumption there is a net $\{y_\beta\}\subseteq X$ that $y_\beta\downarrow 0$ in $X$ and $z\in \{y_\beta\}^l$. By Proposition 5.1 of \cite{3}, $y_\beta\downarrow	0$ in $Y$. Therefore, there is a net $\{y_\beta\}\subseteq Y$ that $y_\beta\downarrow 0$ and for each $\beta$ there is an $\alpha_0$ that for each $\alpha\geq \alpha_0$, $\{\{\pm(x_\alpha - x)\}^u,y\}^l\subseteq \{y_\beta\}^l$ for all $0\leq y \in Y$. So $x_\alpha\xrightarrow{suo}x$ in $Y$.

	Consequently, let $x_\alpha \xrightarrow{suo}x$ in $Y$. Therefore, there exists net $\{y_\beta\}\subseteq Y$ such that $y_\beta \downarrow 0$ and for every $\beta$ there exists $\alpha_0$ such that  $\{\{\pm(x_\alpha - x)\}^u,y\}^l\subseteq \{y_\beta\}^l$, whenever $\alpha \geq \alpha_0$ and for all $0\leq y \in Y$. Since $X$ is order dense in $Y$, therefore $X$ is majorizing in $Y$. Hence for each $\beta$, there is a $z\in X$ such that $y_\beta \leq z$. We define $ z_\beta = \{z\in X: y_\beta \leq z\}$. $z_\beta \downarrow 0$ in $X$. It is obvious that $\{y_\beta\}^l \subseteq \{z_\beta\}^l$. Therefore, for each $\beta$ there is an $\alpha_0$ that for each $\alpha\geq \alpha_0$, $\{\{\pm(x_\alpha - x)\}^u,y\}^l\subseteq \{z_\beta\}^l$ for all $0\leq y \in X$.
\end{proof}
\begin{corollary}
	Let $X$ be a pre-Riesz space with vector lattice cover $(Y,i)$ and $\{x_\alpha\}\subseteq X$ is $semi$-$uo$-null in $X$. By Theorem \ref{55}, $x_\alpha\xrightarrow{uo}0$ in $Y$. If $Y$ has a weak unit $u$ and $u\in X$, then by Lemma 3.2 of \cite{4b}, $x_\alpha\xrightarrow{suo}0$ in $X$ iff there is a net $\{y_\beta\}$, possibly over a different index set, such that $y_\beta \downarrow 0$ and for every $\beta$ there exists $\alpha_0$ such that
	$\{\{\pm(x_\alpha - x)\}^u,u\}^l\subseteq \{y_\beta\}^l$, whenever $\alpha \geq \alpha_0$.
\end{corollary}
\begin{theorem}\label{54}
	Let $X$ be a pre-Riesz space with vector lattice cover $(Y,i)$. Then the following assertions are true.
	\end{theorem}
\begin{enumerate}
	\item  $\{x_\alpha\}\subseteq X$ is $semi$-$uo$-null iff $\{i(x_\alpha)\}$ is $semi$-$uo$-null in $Y$.\label{ii}
	\item If the sequence $\{x_n\}\subseteq X$ is disjoint, then $x_n\xrightarrow{suo}0$ in $X$.
	
	Moreover, if  $X$ is normed space and
	\item  $Y, Y^\prime$ have order continuous norm and $\{x_\alpha\}\subseteq X$ is $semi$-$uo$-null, then $\{x_\alpha\}$ is $w$-null.
	\item $\{x_\alpha\}\subseteq X$ is order bounded, $x_\alpha\xrightarrow{suo}0$ and $Y$ is a Banach lattice with order continuous norm, then $\{x_\alpha\}$ is norm-null.
\end{enumerate}
\begin{proof}
	\begin{enumerate}
		\item Let $\{x_\alpha\}\subseteq X$ is $semi$-$uo$-null in $X$. Therefore, there is a net $\{y_\beta\}$, possibly over a different index set, such that $y_\beta \downarrow 0$ and for every $\beta$ there exists $\alpha_0$ such that
		$\{\{\pm(x_\alpha )\}^u,y\}^l\subseteq \{y_\beta\}^l$, whenever $\alpha \geq \alpha_0$ and for all $0\leq y \in X$. By Lemma 1 of \cite{2}, $i(y_\beta)\downarrow 0$ in $i(X)$. It is obvious that $\{i(x_\alpha)\}$ is $semi$-$uo$-null in $i(X)$. Since $i(X)$ is order dense in $Y$, then by Proposition \ref{55}, $\{i(x_\alpha)\}$ is $semi$-$uo$-null in $Y$.
		
		Conversely, let $\{x_\alpha\}\subseteq X$ and $x_\alpha\xrightarrow{suo}0$ in $Y$. Hence $x_\alpha\xrightarrow{uo}0$ in $Y$ and therefore it is $semi$-$uo$-null in $i(X)$. Therefore, there is a net $\{i(y_\beta)\}$, possibly over a different index set, such that $i(y_\beta) \downarrow 0$ and for every $\beta$ there exists $\alpha_0$ such that
		$$\{\{\pm i(x_\alpha )\}^u,i(y)\}^l\subseteq \{i(y_\beta)\}^l,$$ whenever $\alpha \geq \alpha_0$ and for all $0\leq i(y) \in i(X)$. We have $i(\{\{\pm (x_\alpha )\}^u,y\}^l) =
		 \{\{\pm i(x_\alpha )\}^u,i(y)\}^l\subseteq \{i(y_\beta)\}^l = i(\{y_\beta\}^l).$ Since $i$ is a bipositive operator, therefore $\{\{\pm (x_\alpha )\}^u,y\}^l\subseteq \{y_\beta\}^l$ and so $x_\alpha\xrightarrow{suo}0$ in $X$.
\item	Since $X$ is an order dense subspace of $Y$, by Proposition 5.9 of \cite{3}, $\{x_n\}$ is disjoint in $Y$. By Corollary 3.6 of \cite{4}, $x_n\xrightarrow{uo}0$ in $Y$. Therefore, $x_n\xrightarrow{suo}0$ in $i(X)$. By \ref{ii}, $x_n\xrightarrow{suo}0$ in $X$.
\item Let $\{x_\alpha\}\subseteq X$ be $semi$-$uo$-null. Then by Proposition \ref{55}, $x_\alpha\xrightarrow{uo}0$ in $Y$. Since $Y$ and  $Y^\prime$ have order continuous norm, then by Theorem 5 of \cite{w}, $x_\alpha\xrightarrow{w}$ in $Y$. By Theorem 3.6 of \cite{ru}, $x_\alpha\xrightarrow{w}0$ in $X$.
\item We have $x_\alpha\xrightarrow{uo}0$ in $Y$. Since $Y$ is a Banach lattice with order continuous norm and the net $\{x_\alpha\}$ is almost order bounded in $Y$, therefore by Proposition 3.7 of \cite{4b}, $\{x_\alpha\}$ is norm-null in $Y$ and so is norm-null in $X$.
\end{enumerate}
\end{proof}
\begin{definition}
	A subset $J$ of ordered vector space $X$ is said to be $semi$-$uo$-closed, if $\{x_\alpha\}\subseteq J$ with $x_\alpha\xrightarrow{suo}x$ implies $x\in J$.
\end{definition}
\begin{remark}
	Let $J\subseteq X$ be a $semi$-$uo$-closed set and $\{x_\alpha\}\subseteq J$ that $x_\alpha\xrightarrow{\tilde{o}}x$. It is obvious that $x_\alpha\xrightarrow{suo}x$. Since $J$ is $semi$-$uo$-closed, therefore $x\in J$. So $J$ is $\tilde{o}$-closed.
\end{remark}
\begin{theorem}\label{kk}
	The following assertions are true.
\end{theorem}
\begin{enumerate}
\item	Let $X$ be an order dense subspace of an ordered vector space $Y$. If $J \subseteq Y$ is $semi$-$uo$-closed in $Y$,  then $J \cap X$ is $semi$-$uo$-closed in $X$.
\item Let $X$ be a pre-Riesz space with vector lattice cover $(Y,i)$.  $J\subseteq X$ is $semi$-$uo$-closed iff $i(J)$ is $semi$-$uo$-closed in $i(X)$.
\end{enumerate}
\begin{proof}
	\begin{enumerate}
		\item Let $\{x_\alpha\}\subseteq J\cap X$ and $x_\alpha\xrightarrow{suo}x$. There exists a net $\{y_\beta\}\subseteq X$ that $y_\beta\downarrow 0 $ and for each $\beta$ there exists $\alpha_0$ that  $\{\{\pm(x_\alpha - x)\}^u,y\}^l\subseteq \{y_\beta\}^l$, whenever $\alpha \geq \alpha_0$ and for all $0\leq y \in X$. Since $X$ is order dense in $Y$, then for each $z\in Y$, $\{\{\pm(x_\alpha - x)\}^u,z\}^l\subseteq \{\{\pm(x_\alpha - x)\}^u,y\}^l\subseteq \{y_\beta\}^l$ and $y_\beta \downarrow 0$ in $Y$. Because $\{x_\alpha\}\subseteq J$ and $J$ is $semi$-$uo$-closed in $Y$, therefore $x\in J$. 
		\item Let $\{x_\alpha\}\subseteq J$ and $x_\alpha\xrightarrow{suo}x$. It is obvious that $i(x_\alpha)\xrightarrow{suo}i(x)$. Since $i(J)$ is $semi$-$uo$-closed, therefore $i(x)\in i(J)$. So $x\in J$.
		
		Conversely, let $\{i(x_\alpha)\}\subseteq i(J)$ and $i(x_\alpha)\xrightarrow{suo}i(x)$. There exists a net $\{i(y_\beta)\}\subseteq i(X)$ that $i(y_\beta)\downarrow0 $ and for each $\beta$ there exists $\alpha_0$ that  $\{\{\pm(i(x_\alpha - x)\}^u,i(y)\}^l\subseteq \{i(y_\beta)\}^l$, whenever $\alpha \geq \alpha_0$ and for all $0\leq i(y) \in i(X)$. By Lemma 1 of \cite{2}, $y_\beta\downarrow 0$ in $X$. Therefore, $x_\alpha\xrightarrow{suo}x$. Since $\{x_\alpha\}\subseteq J$ and $J$ is $semi$-$uo$-closed, hence $x\in J$ and therefore $i(x)\in i(J)$. 
	\end{enumerate}
\end{proof}
\begin{theorem}
	Let $J$ be a $semi$-$uo$-closed subspace in Archimedean ordered vector space $X$  and $\{x_\alpha\}\subseteq J$. $x_\alpha\xrightarrow{suo}x$ in $J$ iff $x_\alpha\xrightarrow{suo}x$ in $X$.
\end{theorem}
\begin{proof}
	Let $\{x_\alpha\}\subseteq J$ and $x_\alpha\xrightarrow{suo}x$ in $X$. Since $J$ is $semi$-$uo$-closed it is obvious that $x_\alpha\xrightarrow{suo}x$ in $J$.
	
	Conversely, let $\{x_\alpha\}\subseteq J$ and $x_\alpha\xrightarrow{suo}x$ in $J$. Therefore, there is a net $\{y_\beta\}\subseteq J$, possibly over a different index set, such that $y_\beta \downarrow 0$ in $J$ and for every $\beta$ there exists $\alpha_0$ such that
	$\{\{\pm(x_\alpha-x )\}^u,y\}^l\subseteq \{y_\beta\}^l$, whenever $\alpha \geq \alpha_0$ and for all $0\leq y \in B$. It is obvious that $y_\beta\downarrow $ in $X$. Since $X$ is Archimedean, that there exists a $z$ that  $y_\beta\downarrow z$ in $X$. It is clear that $y_\beta\xrightarrow{\tilde{o}}z$ in $X$ and so $y_\beta\xrightarrow{suo}z$ in $X$. Therefore, $y_\beta\xrightarrow{suo}z$ in $J$. Since by Lemma \ref{eli}, $semi$-$uo$-limits are unique, therefore $z=0$. Hence $y_\beta\downarrow 0$ in $X$. Let $k\in \{\{\pm(x_\alpha - x )\}^u,y\}^l$ for all $0\leq y\in X$. It is clear that $k\in \{\{\pm(x_\alpha - x )\}^u,y^\prime\}^l$ for all $0\leq y^\prime \in J$. So 
	$\{\{\pm(x_\alpha - x )\}^u,y\}^l\subseteq \{y_\beta\}^l$ for all $0\leq y \in X$. Therefore, $x_\alpha\xrightarrow{suo}x$ in $X$.
\end{proof}
\begin{proposition}
	A solid subset $J$ of an ordered vector space $X$ is $semi$-$uo$-closed iff $\{x_\alpha\}\subseteq J$and $0\leq x_\alpha \uparrow x$  imply $x\in J$.
\end{proposition}
\begin{proof}
	Let $\{x_\alpha\}\subseteq J$ and $0\leq x_\alpha \uparrow x$. It is clear that $x_\alpha\xrightarrow{o}x$ and therefore $x_\alpha\xrightarrow{suo}x$. Because $J$ is $semi$-$uo$-closed imply $x\in J$.
	
Conversely, let a net $\{x_\alpha\}\subseteq J$ and $x_\alpha\xrightarrow{suo}x$. Therefore, there is a net $\{y_\beta\}\subseteq X$, possibly over a different index set, such that $y_\beta \downarrow 0$ in $X$ and for every $\beta$ there exists $\alpha_0$ such that
$\{\{\pm(x_\alpha-x )\}^u,y\}^l\subseteq \{y_\beta\}^l$, whenever $\alpha \geq \alpha_0$ and for all $0\leq y \in X$.	It is clear that $\{\pm x_\alpha\}^u \subseteq \{\pm (x - y_\beta)\}^u$. Because $J$ is solid, we have $\{x - y_\beta\}\subseteq J$. It is obvious that  $0\leq x- y_\beta \uparrow x$. It follows that $x\in J$. Hence $J$ is $semi$-$uo$-closed.
\end{proof}
\begin{definition}
	Let $B$ be an ideal in ordered vector space $X$. $B$ is said to be a $semi$-$uo$-band in $X$ if $B$ is $semi$-$uo$-closed in $X$.
\end{definition}
\begin{corollary}
	\begin{enumerate}
		\item Let $X$ be an order dense subspace of directed ordered vector space $Y$. If $J\subseteq Y$ is $semi$-$uo$-band in $Y$, then by Theorem \ref{kk}, $J\cap X$ is $semi$-$uo$-closed in $X$ and by Proposition 5.3 of \cite{3}, $J\cap X$ is an ideal in $X$. Therefore, $J\cap X$ is a $semi$-$uo$-band in $X$.
		
		 Let $X$ be a pre-Riesz space with vector lattice cover $(Y,i)$. 
		 \item It is clear that $J\subseteq X$ is an ideal in $X$ iff $i(J)$ is an ideal in $i(X)$ and by Theorem \ref{kk},  $J\subseteq X$ is $semi$-$uo$-closed in $X$ iff $i(J)$ is $semi$-$uo$-closed in $i(X)$. So, $J\subseteq X$ is a $semi$-$uo$-band in $X$ iff $i(J)$ is $semi$-$uo$-band in $i(X)$.
	\end{enumerate}
\end{corollary}
\begin{proposition}\label{kjkj}
Let $X$ be a vector lattice. If an ideal $B\subseteq X$ is a $semi$-$uo$-band in $X$, then it is a band in $X$.
\end{proposition}
\begin{proof}
	Let an ideal $B$ be a $semi$-$uo$-band in $X$ and $\{x_\alpha\}\subseteq X$ be a net that $x_\alpha\xrightarrow{o}x$. Therefore, $x_\alpha\xrightarrow{uo}x$ in $X$. Since $X$ is a vector lattice, then by Proposition \ref{dd}, $x_\alpha\xrightarrow{suo}x$ in $X$. By assumption, $x\in B$. So $B$ is a band in $X$.
\end{proof}

\end{document}